\documentclass{amsart}
\usepackage{esint}
\usepackage{amsfonts}
\usepackage{amssymb}
\usepackage{xcolor}
\usepackage{url, hyperref}

\newtheorem{theorem}{Theorem}[section]

\newtheorem{corollary}[theorem]{Corollary}

\newtheorem{lemma}[theorem]{Lemma}

\numberwithin{equation}{section}

\begin{document}
\title[Comparison theorems for three-manifolds]{Comparison theorems for
three-dimensional manifolds with scalar curvature bound}
\author{Ovidiu Munteanu and Jiaping Wang}

\begin{abstract}
Two sharp comparison results are derived for three-dimensional complete
noncompact manifolds with scalar curvature bounded from below. The first one
concerns the Green's function. When the scalar curvature is nonnegative, it
states that the rate of decay of an energy quantity over the level set is
strictly less than that of the Euclidean space unless the manifold itself is
isometric to the Euclidean space. The result is in turn converted into a
sharp area comparison for the level set of the Green's function when in
addition the Ricci curvature of the manifold is assumed to be asymptotically
nonnegative at infinity. The second result provides a sharp upper bound of
the bottom spectrum in terms of the scalar curvature lower bound, in
contrast to the classical result of Cheng which involves a Ricci curvature
lower bound.
\end{abstract}

\address{Department of Mathematics, University of Connecticut, Storrs, CT
06268, USA}
\email{ovidiu.munteanu@uconn.edu}
\address{School of Mathematics, University of Minnesota, Minneapolis, MN
55455, USA}
\email{jiaping@math.umn.edu}
\maketitle

\section{Introduction}

The classical Laplacian comparison theorem \cite[Chapter I]{SY94} states
that the Laplacian of a geodesic distance function is at most that of the
corresponding space form under a Ricci curvature lower bound. It is obvious
that such a result is no longer true under a scalar curvature lower bound.
The purpose of this note is to search for suitable alternatives for complete
three-dimensional manifolds.

Our first result deals with the case of nonnegative scalar curvature. Recall
that a complete manifold is called nonparabolic if it admits a positive
Green's function \cite[Chapter 20]{L}. It is well-known that in this case
the minimal positive Green's function $G(x,y)$ may be obtained as the limit
of the Dirichlet Green's function of a sequence of compact exhaustive
domains of the manifold. Then

\begin{equation*}
\Delta _{x}G\left( x,y\right) =-\delta \left( x,y\right),
\end{equation*}%
\begin{equation*}
G\left( x,y\right) =G\left( y,x\right) >0
\end{equation*}
and 
\begin{equation*}
\liminf_{y\to \infty} G(x,y)=0.
\end{equation*}
Throughout the paper we fix $p\in M$ and let

\begin{equation*}
G\left( x\right) =G\left( p,x\right).
\end{equation*}%
We also use the following notations to denote the level and sublevel sets of 
$G(x).$ 
\begin{eqnarray*}
L\left( a,b\right) &=&\left\{ x\in M:\ a<G\left( x\right) <b\right\} \\
l\left( t\right) &=&\left\{ x\in M:\ G\left( x\right) =t\right\}.
\end{eqnarray*}

We have the following sharp comparison theorem concerning the minimal
positive Green's function.

\begin{theorem}
\label{T3} Let $\left( M,g\right) $ be a complete noncompact
three-dimensional manifold with nonnegative scalar curvature. Assume that $M$
has one end and its first Betti number $b_{1}\left( M\right) =0.$ If $M$ is
nonparabolic and the minimal positive Green's function $G\left( x\right)
=G\left( p,x\right) $ satisfies $\lim_{x\rightarrow \infty }G(x)=0,$ then

\begin{equation*}
\frac{d}{dt}\left( \frac{1}{t}\int_{l\left( t\right) }\left\vert \nabla
G\right\vert ^{2}-4\pi t\right) \leq 0
\end{equation*}%
for all $t>0.$ Moreover, equality holds for some $T>0$ if and only if the
super level set $\left\{x\in M, G(x)>T\right\}$ is isometric to a ball in
the Euclidean space $\mathbb{R}^3.$
\end{theorem}

Some remarks are in order. First, the conclusion may be restated as

\begin{equation*}
\frac{d}{dt}\left( \frac{1}{t}\int_{l\left( t\right) }\left\vert \nabla
G\right\vert ^{2} \right)\leq \frac{d}{dt}\left( \frac{1}{t}\int_{\bar{l}%
\left( t\right) }\left\vert \nabla \bar{G}\right\vert ^{2} \right)
\end{equation*}%
for all $t>0,$ where $\bar{G}(\bar{x})=\frac{1}{4\pi |\bar x|}$ is the
Green's function of $\mathbb{R}^3$ and $\bar{l}(t)$ the level set of $\bar{G}%
.$ As such, it may be viewed as a comparison of the decay rate concerning
the energy quantity $\frac{1}{t}\int_{l\left( t\right) }\left\vert \nabla
G\right\vert ^{2}$ of $M$ with that of $\mathbb{R}^3.$ Second, the fact that 
$\lim_{x\to \infty} G(x)=0$ together with the topological information of $M$
is to ensure that the level set $l(t)$ is compact and connected. Without
those assumptions, one may work with the Dirichlet Green's function of an
arbitrary bounded domain instead. The resulting conclusion now depends on
the number of components of $l(t)$ as well. Third, it is unclear to us if an
analogous conclusion holds in the higher dimension, though our proof is very
dimension specific.

Under the additional assumption that the Ricci curvature is nonnegative at
infinity, the preceding result may be converted into an area comparison
theorem for the level sets of the Green's function.

\begin{corollary}
\label{C1} Let $\left( M,g\right) $ be a complete noncompact
three-dimensional manifold with nonnegative scalar curvature and
asymptotically nonnegative Ricci curvature, that is, 
\begin{equation*}
\liminf_{x\rightarrow \infty }\mathrm{Ric}\left( x\right) \geq 0.
\end{equation*}%
Assume that $M$ has one end and its first Betti number $b_{1}\left( M\right)
=0.$ If $M$ is nonparabolic and the minimal Green's function $G\left(
x\right) =G\left( p,x\right) $ satisfies $\lim_{x\rightarrow \infty }G(x)=0,$
then

\begin{equation*}
\int_{l\left( t\right) }\left\vert \nabla G\right\vert ^{2}\leq 4\pi t^{2}
\end{equation*}%
and 
\begin{equation*}
\mathrm{Area}\left( l\left( t\right) \right) \geq \frac{1}{4\pi t^{2}}
\end{equation*}%
for all $t>0.$ Moreover, if equality holds for some $T>0,$ then the super
level set $\left\{ G>T\right\} $ is isometric to a ball in $\mathbb{R}^3.$
\end{corollary}

Note that in the case that $M$ has nonnegative Ricci curvature, its minimal
positive Green's function $G$ always satisfies $\lim_{x\to \infty}G(x)=0$ by 
\cite{LY} and the number of ends is necessarily one due to the
Cheeger-Gromoll splitting theorem \cite{CG}. Therefore, Theorem \ref{T3} and
Corollary \ref{C1} are both applicable to the universal cover of $M.$

Theorem \ref{T3} is motivated by the work of Colding \cite%
{Co} and Colding-Minicozzi \cite{CM1}, where monotonicity formulas for
functionals of the form

\begin{equation*}
w_{p}\left( r\right) =\frac{1}{r^{n-1}}\int_{b=r}\left\vert \nabla
b\right\vert ^{p}
\end{equation*}%
are derived for $n$-dimensional manifolds with nonnegative Ricci curvature,
where the function $b=G^{-\frac{1}{n-2}}.$ So Theorem \ref{T3} concerns $%
w_{2}$ for dimension $n=3,$ while the exponent $p=3$ in \cite{Co}, and more
generally $p\geq \frac{2n-3}{n-1}$ in \cite{CM1}, for all dimension $n.$
These monotonicity results have been applied to the study of uniqueness of
the tangent cones for Ricci flat manifolds with Euclidean volume growth \cite%
{CM2}. We refer the readers to \cite{CM} for an exposition on monotonicity
formulas in geometric analysis, and \cite{AFM} for their applications to
Willmore type inequalities.

Our second result concerns the bottom spectrum. Recall that the bottom
spectrum $\lambda _{1}\left( M\right) $ is characterized by 
\begin{equation*}
\lambda _{1}\left( M\right) =\inf_{f\in C_{0}^{\infty }\left( M\right) }%
\frac{\int_{M}\left\vert \nabla f\right\vert ^{2}}{\int_{M}f^{2}}.
\end{equation*}%
According to Cheng's theorem \cite{C}, for an $n$-dimensional complete
manifold $M,$

\begin{equation*}
\lambda _{1}\left( M\right) \leq \frac{\left( n-1\right) ^{2}}{4}K
\end{equation*}%
if the Ricci curvature of $\left( M,g\right) $ satisfies $\mathrm{Ric}\geq
-\left( n-1\right) K$ for some nonnegative constant $K.$ By considering the
example of the form $M=\mathbb{H}^{2}\times \mathbb{S}^{n-2}(r),$ where $%
\mathbb{H}^2$ is the standard hyperbolic plane of sectional curvature $-1$
and $\mathbb{S}^{n-2}(r)$ the sphere of radius $r$ in $\mathbb{R}^{n-1},$
one sees that a direct extension of Cheng's result to the scalar curvature
lower bound is not possible for dimension $n\geq 4.$ Indeed, by choosing $r$
accordingly, the scalar curvature of $M$ can be made as large as one
desires, yet $\lambda_1(M)=\frac{1}{4}.$ However, for $n=3,$ one does have
the following theorem.

\begin{theorem}
\label{T2} Let $\left( M,g\right) $ be a three-dimensional complete
noncompact Riemannian manifold with scalar curvature $S\geq -6K$ on $M$ for
some nonnegative constant $K.$ Suppose that $M$ has finitely many ends and
its first Betti number $b_{1}(M)<\infty .$ Moreover, the Ricci curvature of $%
M$ is bounded from below and the volume $V_{x}(1)$ of unit ball $B_{x}(1)$
satisfies

\begin{equation*}
V_x(1)\geq C(\epsilon)\,\exp\left(-2\sqrt{K+\epsilon}\,r(x)\right)
\end{equation*}%
for every $\epsilon>0$ and all $x\in M,$ where $r(x)$ is the geodesic
distance from $x$ to a fixed point $p.$ Then the bottom spectrum of the
Laplacian satisfies 
\begin{equation*}
\lambda _{1}\left( M\right) \leq K.
\end{equation*}
\end{theorem}

Note that in the case that the Ricci curvature of $M$ is bounded by $\mathrm{%
Ric}\geq -2K,$ the Gromov-Bishop volume comparison theorem \cite[Chapter 2]%
{L} readily implies that volume lower bound for $V_{x}(1)$ holds. So, by
considering the universal cover of $M$ if necessary and modulo the
topological assumption of finitely many ends, Theorem \ref{T2} provides a
faithful generalization of Cheng's result to three-dimensional manifolds
with only scalar curvature lower bound.

\begin{corollary}
Let $\left( M,g\right) $ be a three-dimensional complete noncompact
Riemannian manifold with nonpositive sectional curvature. Assume that the
scalar curvature $S$ is bounded by

\begin{equation*}
S\geq -6K\text{ \ on }M
\end{equation*}
for some nonnegative constant $K.$ Then 
\begin{equation*}
\lambda _{1}\left( M\right) \leq K.
\end{equation*}
\end{corollary}

This is because Theorem \ref{T2} is applicable to the universal cover $%
\tilde{M}$ of $M$ as $\tilde{M}$ is a Cartan-Hadamard manifold with bounded
curvature. Since the bottom spectrum satisfies $\lambda_1(M)\leq \lambda_1(%
\tilde{M}),$ the corollary follows.

Both Theorem \ref{T3} and Theorem \ref{T2} are proved by working with the
minimal positive Green's function $G$ of $M.$ The idea of using Green's
function to bound the bottom spectrum was introduced by the first author in 
\cite{M}. Roughly speaking, one takes a test function $f=\left\vert \nabla
G\right\vert ^{\frac{1}{2}}\phi$ with a carefully chosen cut-off function $%
\phi$ which in turn depends on $G.$ The proofs of both theorems hinge on
manipulating the Bochner formula for the Green's function.

\begin{equation*}
\Delta \left\vert \nabla G\right\vert =\left( \left\vert G_{ij}\right\vert
^{2}-\left\vert \nabla \left\vert \nabla G\right\vert \right\vert
^{2}\right) \left\vert \nabla G\right\vert ^{-1}+\mathrm{Ric}\left( \nabla
G,\nabla G\right) \left\vert \nabla G\right\vert ^{-1}.
\end{equation*}%
A crucial point is to rewrite the Ricci curvature term by mimicking a trick
originated from the work of Schoen and Yau \cite{SY, SY1, SY2} on stable
minimal surfaces in three-dimensional manifolds.

\begin{equation*}
\mathrm{Ric}\left( \nabla G,\nabla G\right) \left\vert \nabla G\right\vert
^{-2}=\frac{1}{2}S-\frac{1}{2}S_{l\left( r\right) }+\frac{1}{\left\vert
\nabla G\right\vert ^{2}}\left( \left\vert \nabla \left\vert \nabla
G\right\vert \right\vert ^{2}-\frac{1}{2}\left\vert \nabla ^{2}G\right\vert
^{2}\right),
\end{equation*}%
where $S_{l\left( r\right) }$ denotes the scalar curvature of the level set $%
l\left(r\right).$ One can then proceed by integrating the formula over the
level sets and applying the Gauss-Bonnet theorem. We note that this kind of
idea has been exploited recently in \cite{S} and \cite{BKHS} as well.

In order to make the argument work, however, we need to ensure that the
level sets $l(t)$ of $G$ are compact with controlled number of components.
This is where all the extra assumptions are used to show the following.

\begin{eqnarray}
\lim_{x\rightarrow \infty }G\left( x\right) &=&0,  \label{s1} \\
\#\mathrm{Conn}\left( l\left( t\right) \right) &\leq &A  \label{s2}
\end{eqnarray}%
for all $t,$ where $A$ is a fixed constant and $\#\mathrm{Conn}\left( l(t)
\right)$ the number of connected components of the level set $l(t).$

Ideally, one would like to prove Theorem \ref{T2} under the sole assumption
that the scalar curvature $S\geq -6K.$ Another natural question is what
happens if $\lambda_1(M)$ achieves its maximum value $K$ in Theorem \ref{T2}%
. In the case of the aforementioned Cheng's theorem, there are rigidity
results \cite{LW1, LW2}.

The study of scalar curvature has a long history with many significant
results. The recent lecture notes \cite{G} and the survey paper \cite{So}
are good sources for the state of the affairs and references.

The structure of the paper is at follows. After collecting some preliminary
results in Section \ref{sect2}, we supply the proofs of Theorem \ref{T3} and
Theorem \ref{T2} in Section \ref{sect3} and Section \ref{sect4},
respectively.

\section{Preliminaries\label{sect2}}

In this section we make some preparations for proving Theorem \ref{T3} and
Theorem \ref{T2}. Let us start with the following result. It relies on an
idea from Schoen-Yau's work on minimal surfaces \cite{SY,SY1, SY2} and
appears as Lemma 4.1 in \cite{BKHS}. We include details here for
completeness.

\begin{lemma}
\label{RicS}Let $\left( M,g\right) $ be a three-dimensional complete
noncompact Riemannian manifold with scalar curvature $S$ and $u$ a harmonic
function on $M.$ Then on each regular level set $l\left( r\right) $ of $u,$

\begin{equation*}
\mathrm{Ric}\left( \nabla u,\nabla u\right) \left\vert \nabla u\right\vert
^{-2}=\frac{1}{2}S-\frac{1}{2}S_{l\left( r\right) }+\frac{1}{\left\vert
\nabla u\right\vert ^{2}}\left( \left\vert \nabla \left\vert \nabla
u\right\vert \right\vert ^{2}-\frac{1}{2}\left\vert \nabla ^{2}u\right\vert
^{2}\right),
\end{equation*}%
where $S_{l\left( r\right) }$ denotes the scalar curvature of $%
l\left(r\right).$
\end{lemma}

\begin{proof}
On a regular level set $l\left( r\right) $ of $u,$ its unit normal vector is
given 
\begin{equation*}
e_{1}=\frac{\nabla u}{\left\vert \nabla u\right\vert }.
\end{equation*}%
Choose $\left\{ e_{a}\right\} _{a=2,3},$ unit vectors tangent to $%
l\left(r\right),$ such that $\left\{ e_{1},e_{2},e_{3}\right\} $ forms a
local orthonormal frame on $M.$ Since $u$ is harmonic, the second
fundamental form and the mean curvature of $l\left( r\right) $ are given by

\begin{equation*}
h_{ab}=\frac{u_{ab}}{\left\vert \nabla u\right\vert }\text{ \ and }H=-\frac{%
u_{11}}{\left\vert \nabla u\right\vert }, \text{ \ respectively. }
\end{equation*}%
By the Gauss curvature equation, we have

\begin{equation*}
S_{l\left( r\right) }=S-2R_{11}+H^{2}-h^{2}.
\end{equation*}%
Therefore,

\begin{eqnarray*}
2\mathrm{Ric}\left( \nabla u,\nabla u\right) \left\vert \nabla u\right\vert
^{-2} &=&2R_{11} \\
&=&S-S_{l\left( r\right) }+\frac{1}{\left\vert \nabla u\right\vert ^{2}}%
\left( \left\vert u_{11}\right\vert ^{2}-\left\vert u_{ab}\right\vert
^{2}\right) \\
&=&S-S_{l\left( r\right) }+\frac{1}{\left\vert \nabla u\right\vert ^{2}}%
\left( 2\left\vert \nabla \left\vert \nabla u\right\vert \right\vert
^{2}-\left\vert u_{ij}\right\vert ^{2}\right),
\end{eqnarray*}%
where we have used the fact that

\begin{equation*}
\left\vert \nabla \left\vert \nabla u\right\vert \right\vert ^{2}
=\left\vert u_{11}\right\vert ^{2}+\left\vert u_{1a}\right\vert ^{2}
\end{equation*}
and 
\begin{equation*}
\left\vert u_{ij}\right\vert ^{2} =\left\vert u_{11}\right\vert
^{2}+2\left\vert u_{1a}\right\vert ^{2}+\left\vert u_{ab}\right\vert ^{2}.
\end{equation*}%
This proves the result.
\end{proof}

We will also use the following well known Kato inequality for harmonic
functions.

\begin{lemma}
\label{Kato} Let $\left( M,g\right) $ be a three-dimensional complete
noncompact Riemannian manifold and $u$ a harmonic function on $M.$ Then

\begin{equation*}
\left\vert \nabla ^{2}u\right\vert ^{2}\geq \frac{3}{2}\left\vert \nabla
\left\vert \nabla u\right\vert \right\vert ^{2}\text{ \ on }M.
\end{equation*}
\end{lemma}

\begin{proof}
It suffices to prove this at points where $\left\vert \nabla u\right\vert
\neq 0.$ Let

\begin{equation*}
e_{1}=\frac{\nabla u}{\left\vert \nabla u\right\vert }
\end{equation*}%
and choose $\left\{ e_{a}\right\} _{a=2,3}$ such that $\left\{
e_{1},e_{2},e_{3}\right\} $ is a local orthonormal frame on $M.$ Then

\begin{eqnarray*}
\left\vert u_{ij}\right\vert ^{2} &=&\left\vert u_{11}\right\vert
^{2}+2\left\vert u_{1a}\right\vert ^{2}+\left\vert u_{ab}\right\vert ^{2} \\
&\geq &\left\vert u_{11}\right\vert ^{2}+2\left\vert u_{1a}\right\vert ^{2}+%
\frac{1}{2}\left\vert u_{22}+u_{33}\right\vert ^{2} \\
&=&\frac{3}{2}\left\vert u_{11}\right\vert ^{2}+2\left\vert
u_{1a}\right\vert ^{2},
\end{eqnarray*}%
where in the last line we have used the fact that $u$ is harmonic. Therefore,

\begin{eqnarray*}
\left\vert u_{ij}\right\vert ^{2} &\geq &\frac{3}{2}\left( \left\vert
u_{11}\right\vert ^{2}+\left\vert u_{1a}\right\vert ^{2}\right) \\
&=&\frac{3}{2}\left\vert \nabla \left\vert \nabla u\right\vert \right\vert
^{2}.
\end{eqnarray*}
\end{proof}

We also need the following topological lemma concerning the number of
components of the level sets of a proper Green's function. The proof is
inspired by \cite{LT}.

\begin{lemma}
\label{LT} Let $\left( M,g\right) $ be a complete noncompact Riemannian
manifold with $k$ ends and finite first Betti number $b_1(M).$ Assume that $%
\left( M,g\right) $ is nonparabolic and its minimal positive Green's
function $G$ satisfies $\lim_{x\to \infty} G(x)=0.$ Then there exists $t_0>0$
such that the level set $l(t)$ of $G$ has exactly $k$ components for all $%
t\leq t_0.$ In the case that $M$ has only one end and the first Betti number 
$b_{1}\left( M\right) =0,$ the level set $l\left( t\right) $ is connected
for all $t>0.$
\end{lemma}

\begin{proof}
We first claim that 
\begin{equation}
L\left(a,\infty \right) \text{ \ is connected}  \label{v3}
\end{equation}%
for all $a>0.$ Indeed, if this is not true, then there exists a connected
component $L_{0}$ of $L\left( a,\infty \right) $ which does not contain the
pole $p$ of $G.$ Moreover, $L_{0}$ is bounded by the fact that $\lim_{x\to
\infty}G(x)=0.$ Then the harmonic function $G$ on $L_{0}$ must achieve its
maximum in the interior of $L_{0},$ which is a contradiction.

We also note that $L(0, a)$ contains no bounded components for any $a>0.$
Otherwise, on such a bounded component the function $G$ would achieve its
minimum at an interior point.

Since the first Betti number of $M$ is finite, we may choose $t_{0}$
sufficiently small so that all the representatives of $H_{1}\left( M\right) $
lie in $L\left( t_{0},\infty \right).$ By arranging $t_0$ to be even smaller
if necessary, we may assume that $L(0, t)$ has exactly $k$ components for $%
t\leq t_0$ by the fact that it contains no bounded components and that $M$
has $k$ ends.

For $t<t_{0}$ and $\rho >0$ such that $t+\rho<t_0,$ since $M=L\left(
t,\infty \right) \cup L\left(0,t+\rho \right),$ we have the Mayer-Vietoris
sequence

\begin{gather*}
H_{1}\left( L\left( t,\infty \right) \right) \oplus H_{1}\left( L\left(
0,t+\rho \right) \right) \rightarrow H_{1}\left( M\right) \rightarrow
H_{0}\left( L\left( t,t+\rho \right) \right) \\
\rightarrow H_{0}\left( L\left( t,\infty \right) \right) \oplus H_{0}\left(
L\left( 0,t+\rho \right) \right) \rightarrow H_{0}(M).
\end{gather*}%
Note that $H_{1}\left( L\left( 0,t+\rho \right) \right) $ is trivial because
all the representatives of $H_{1}\left( M\right) $ lie inside $L\left(
t_{0},\infty \right).$ In view of (\ref{v3}), we therefore conclude that

\begin{equation*}
H_{0}\left( L\left( t,t+\rho \right) \right) =H_{0}\left(L\left( 0,t+\rho
\right) \right)= \mathbb{Z}\oplus ..\oplus \mathbb{Z}
\end{equation*}%
with $k$ summands. In other words, $L\left(t,t+\rho \right) $ has $k$
components. Since $\rho >0$ can be arbitrarily small, this proves that

\begin{equation}
l\left( t\right) \text{ has at most }k\text{ components}  \label{v6}
\end{equation}%
for $t\leq t_{0}.$ It is obvious that $l(t)$ can not have fewer than $k$
components for any $t\leq t_0$ as $L(0, t_0)$ has $k$ components. Therefore, 
$l(t)$ has exactly $k$ components for $t\leq t_0.$

In the special case that $b_{1}\left( M\right) =0,$ the number $t_{0}$ can
be taken arbitrarily large. By (\ref{v6}), one concludes that $l\left(
t\right) $ is connected for all $t>0.$
\end{proof}

\section{Nonnegative scalar curvature\label{sect3}}

In this section, we work with three-dimensional complete manifolds with
nonnegative scalar curvature and prove Theorem \ref{T3} which is restated
below.

\begin{theorem}
\label{Positive} Let $\left( M,g\right) $ be a complete noncompact
three-dimensional manifold with nonnegative scalar curvature. Assume that $M$
has one end and its first Betti number $b_{1}\left( M\right) =0.$ If $M$ is
nonparabolic and the minimal positive Green's function $G\left( x\right)
=G\left( p,x\right) $ satisfies $\lim_{x\rightarrow \infty }G(x)=0,$ then

\begin{equation*}
\frac{d}{dt}\left( \frac{1}{t}\int_{l\left( t\right) }\left\vert \nabla
G\right\vert ^{2}-4\pi t\right) \leq 0
\end{equation*}%
for all $t>0.$ Moreover, equality holds for some $T>0$ if and only if the
super level set $\left\{x\in M, G(x)>T\right\}$ is isometric to a ball in
the Euclidean space $\mathbb{R}^3.$
\end{theorem}

\begin{proof}
Recall that 
\begin{eqnarray*}
l\left( t\right)  &=&\left\{ x\in M:G\left( x\right) =t\right\}  \\
L\left( a,b\right)  &=&\left\{ x\in M:a<G\left( x\right) <b\right\} .
\end{eqnarray*}%
By the assumption that $\lim_{x\rightarrow \infty }G(x)=0,$ the level set $%
l\left( t\right) $ is compact for every $t>0.$ Moreover, since $M$ is
assumed to have one end and $b_{1}\left( M\right) =0,$ Lemma \ref{LT}
implies that $l\left( t\right) $ is connected for all $t>0.$

Consider the function 
\begin{equation*}
w\left( t\right) =\int_{l\left( t\right) }\left\vert \nabla G\right\vert^{2}.
\end{equation*}%
Whenever $l\left( t\right) $ is regular, its mean curvature $H$ is given by

\begin{equation*}
H_{l\left( t\right) }=\frac{\sum_{a}G_{aa}}{\left\vert \nabla G\right\vert }%
=-\frac{G_{11}}{\left\vert \nabla G\right\vert }=-\frac{\left\langle \nabla
\left\vert \nabla G\right\vert ,\nabla G\right\rangle }{\left\vert \nabla
G\right\vert ^{2}},
\end{equation*}%
where $e_{1}=\frac{\nabla G}{\left\vert \nabla G\right\vert }$ and $%
\left\{e_{a}\right\} _{a=2,3}$ are unit tangent vectors on $l\left( t\right) 
$ such that $\left\{ e_{1},e_{2},e_{3}\right\} $ is a local orthonormal
frame on $M.$ It follows that

\begin{eqnarray*}
\frac{dw}{dt}\left( t\right) &=&\int_{l\left( t\right) }\left( \frac{%
\left\langle \nabla \left\vert \nabla G\right\vert ^{2},\nabla
G\right\rangle }{\left\vert \nabla G\right\vert ^{2}}+\frac{H_{l\left(
t\right) }}{\left\vert \nabla G\right\vert }\left\vert \nabla G\right\vert
^{2}\right) \\
&=&\int_{l\left( t\right) }\frac{\left\langle \nabla \left\vert \nabla
G\right\vert ,\nabla G\right\rangle }{\left\vert \nabla G\right\vert }.
\end{eqnarray*}%
Multiplying the equation by $t^{-2}$ we get

\begin{equation}
t^{-2}\frac{dw}{dt}\left( t\right) =\int_{l\left( t\right) }\frac{%
\left\langle \nabla \left\vert \nabla G\right\vert ,\nabla G\right\rangle }{%
\left\vert \nabla G\right\vert }G^{-2}.  \label{m1}
\end{equation}%
On the other hand, by Green's identity we have

\begin{eqnarray*}
&&\int_{L\left( t,T\right) }\left( G^{-2}\Delta \left\vert \nabla
G\right\vert -\left\vert \nabla G\right\vert \Delta G^{-2}\right) \\
&=&\int_{l\left( T\right) }\left( G^{-2}\frac{\left\langle \nabla \left\vert
\nabla G\right\vert ,\nabla G\right\rangle }{\left\vert \nabla G\right\vert }%
-\left\vert \nabla G\right\vert \frac{\left\langle \nabla G^{-2},\nabla
G\right\rangle }{\left\vert \nabla G\right\vert }\right) \\
&&-\int_{l\left( t\right) }\left( G^{-2}\frac{\left\langle \nabla \left\vert
\nabla G\right\vert ,\nabla G\right\rangle }{\left\vert \nabla G\right\vert }%
-\left\vert \nabla G\right\vert \frac{\left\langle \nabla G^{-2},\nabla
G\right\rangle }{\left\vert \nabla G\right\vert }\right) .
\end{eqnarray*}%
Since $G$ is harmonic on $L\left( t,T\right) $,

\begin{equation*}
\left\vert \nabla G\right\vert \Delta G^{-2}=6G^{-4}\left\vert \nabla
G\right\vert ^{3}.
\end{equation*}%
Note that as $x\rightarrow p,$

\begin{equation}
G\left( x\right) \simeq \frac{1}{4\pi }\frac{1}{r\left( x\right) }\text{ and 
}\left\vert \nabla G\right\vert \left( x\right) \simeq \frac{1}{4\pi }\frac{1%
}{r^{2}\left( x\right) }.  \label{m2}
\end{equation}%
Hence,

\begin{equation*}
\lim_{T\rightarrow \infty }\int_{l\left( T\right) }\left( G^{-2}\frac{%
\left\langle \nabla \left\vert \nabla G\right\vert ,\nabla G\right\rangle }{%
\left\vert \nabla G\right\vert }-\left\vert \nabla G\right\vert \frac{%
\left\langle \nabla G^{-2},\nabla G\right\rangle }{\left\vert \nabla
G\right\vert }\right) =0
\end{equation*}%
and 
\begin{equation*}
\int_{L\left( t,\infty \right) }\left\vert G^{-2}\Delta \left\vert \nabla
G\right\vert -\left\vert \nabla G\right\vert \Delta G^{-2}\right\vert
<\infty.
\end{equation*}%
In conclusion, we obtain the following identity.

\begin{eqnarray*}
&&\int_{l\left( t\right) }\left( G^{-2}\frac{\left\langle \nabla \left\vert
\nabla G\right\vert ,\nabla G\right\rangle }{\left\vert \nabla G\right\vert }%
-\left\vert \nabla G\right\vert \frac{\left\langle \nabla G^{-2},\nabla
G\right\rangle }{\left\vert \nabla G\right\vert }\right)  \\
&=&-\int_{L\left( t,\infty \right) }\left( G^{-2}\Delta \left\vert \nabla
G\right\vert -6G^{-4}\left\vert \nabla G\right\vert ^{3}\right) .
\end{eqnarray*}%
Together with (\ref{m1}), we conclude that

\begin{eqnarray}
t^{-2}\frac{dw}{dt}\left( t\right) &=&\int_{l\left( t\right) }G^{-2}\frac{%
\left\langle \nabla \left\vert \nabla G\right\vert ,\nabla G\right\rangle }{%
\left\vert \nabla G\right\vert }  \label{m2.1} \\
&=&-2\int_{l\left( t\right) }G^{-3}\left\vert \nabla G\right\vert
^{2}-\int_{L\left( t,\infty \right) }G^{-2}\Delta \left\vert \nabla
G\right\vert  \notag \\
&&+6\int_{L\left( t,\infty \right) }G^{-4}\left\vert \nabla G\right\vert
^{3}.  \notag
\end{eqnarray}%
Note that by the co-area formula,

\begin{eqnarray*}
\int_{L\left( t,\infty \right) }G^{-4}\left\vert \nabla G\right\vert ^{3}
&=&\int_{t}^{\infty }r^{-4}\int_{l\left( r\right) }\left\vert \nabla
G\right\vert ^{2} \\
&=&\int_{t}^{\infty }r^{-4}w\left( r\right) dr.
\end{eqnarray*}%
Hence, (\ref{m2.1}) can be written as

\begin{eqnarray}
t^{-2}\frac{dw}{dt}\left( t\right) &=&-2t^{-3}w\left( t\right)
+6\int_{t}^{\infty }r^{-4}w\left( r\right) dr  \label{m3} \\
&&-\int_{L\left( t,\infty \right) }G^{-2}\Delta \left\vert \nabla
G\right\vert.  \notag
\end{eqnarray}

We now estimate the last term. Using the Bochner formula

\begin{equation*}
\Delta \left\vert \nabla G\right\vert =\left( \left\vert G_{ij}\right\vert
^{2}-\left\vert \nabla \left\vert \nabla G\right\vert \right\vert
^{2}\right) \left\vert \nabla G\right\vert ^{-1}+\mathrm{Ric}\left( \nabla
G,\nabla G\right) \left\vert \nabla G\right\vert ^{-1},
\end{equation*}%
we have

\begin{eqnarray*}
\int_{l\left( r\right) }\left\vert \nabla G\right\vert ^{-1}\Delta
\left\vert \nabla G\right\vert &=&\int_{l\left( r\right) }\left( \left\vert
G_{ij}\right\vert ^{2}-\left\vert \nabla \left\vert \nabla G\right\vert
\right\vert ^{2}\right) \left\vert \nabla G\right\vert ^{-2} \\
&&+\int_{l\left( r\right) }\mathrm{Ric}\left( \nabla G,\nabla G\right)
\left\vert \nabla G\right\vert ^{-2}.
\end{eqnarray*}%
Applying Lemma \ref{RicS} to $G$ gives

\begin{equation*}
\mathrm{Ric}\left( \nabla G,\nabla G\right) \left\vert \nabla G\right\vert
^{-2}=\frac{1}{2}S-\frac{1}{2}S_{l\left( r\right) }+\left( \left\vert \nabla
\left\vert \nabla G\right\vert \right\vert ^{2}-\frac{1}{2}\left\vert
G_{ij}\right\vert ^{2}\right) \left\vert \nabla G\right\vert ^{-2},
\end{equation*}%
where $S_{l\left( r\right) }$ is the scalar curvature of $l\left( r\right).$
We therefore conclude that

\begin{equation}
\int_{l\left( r\right) }\left\vert \nabla G\right\vert ^{-1}\Delta
\left\vert \nabla G\right\vert =\frac{1}{2}\int_{l\left( r\right) }\left(
\left\vert G_{ij}\right\vert ^{2}\left\vert \nabla G\right\vert
^{-2}+S-S_{l\left( r\right) }\right).  \label{m5}
\end{equation}%
Note that by Lemma \ref{Kato},

\begin{equation*}
\left\vert G_{ij}\right\vert ^{2}\geq \frac{3}{2}\left\vert \nabla
\left\vert \nabla G\right\vert \right\vert ^{2}.
\end{equation*}%
Also, since $l\left( r\right) $ is compact and connected for any $r>0,$ the
Gauss-Bonnet theorem implies that

\begin{equation*}
\int_{l\left( r\right) }S_{l\left( r\right) }=4\pi \chi \left(
l\left(r\right) \right) \leq 8\pi
\end{equation*}%
whenever $r$ is a regular value of $G.$ Therefore, on any regular level set $%
l(r),$ one obtains from (\ref{m5}) that

\begin{equation}
\int_{l\left( r\right) }\left\vert \nabla G\right\vert ^{-1}\Delta
\left\vert \nabla G\right\vert \geq \frac{3}{4}\int_{l\left( r\right)
}\left\vert \nabla \left\vert \nabla G\right\vert \right\vert ^{2}\left\vert
\nabla G\right\vert ^{-2}-4\pi.  \label{m6}
\end{equation}

Observe from (\ref{m1}) that

\begin{eqnarray}
\left\vert w^{\prime }\left( r\right) \right\vert &=&\left\vert
\int_{l\left( r\right) }\frac{\left\langle \nabla \left\vert \nabla
G\right\vert ,\nabla G\right\rangle }{\left\vert \nabla G\right\vert }%
\right\vert  \label{m6'} \\
&\leq &\int_{l\left( r\right) }\left\vert \nabla \left\vert \nabla
G\right\vert \right\vert  \notag \\
&\leq &\left( \int_{l\left( r\right) }\left\vert \nabla \left\vert \nabla
G\right\vert \right\vert ^{2}\left\vert \nabla G\right\vert ^{-2}\right) ^{%
\frac{1}{2}}\left( \int_{l\left( r\right) }\left\vert \nabla G\right\vert
^{2}\right) ^{\frac{1}{2}},  \notag
\end{eqnarray}%
which says that

\begin{equation*}
\int_{l\left( r\right) }\left\vert \nabla \left\vert \nabla G\right\vert
\right\vert ^{2}\left\vert \nabla G\right\vert ^{-2}\geq \frac{\left(
w^{\prime }\right) ^{2}}{w}\left( r\right).
\end{equation*}%
Combining with (\ref{m6}) we conclude that

\begin{equation}
\int_{l\left( r\right) }\left\vert \nabla G\right\vert ^{-1}\Delta
\left\vert \nabla G\right\vert \geq \frac{3}{4}\frac{\left( w^{\prime
}\right) ^{2}}{w}\left( r\right) -4\pi.  \label{m7}
\end{equation}

By (\ref{m7}) and the co-area formula it follows that

\begin{eqnarray}
-\int_{L\left( t,\infty \right) }G^{-2}\Delta \left\vert \nabla G\right\vert
&=&-\int_{t}^{\infty }r^{-2}\int_{l\left( r\right) }\left\vert \nabla
G\right\vert ^{-1}\Delta \left\vert \nabla G\right\vert  \label{m7.1} \\
&\leq &-\frac{3}{4}\int_{t}^{\infty }r^{-2}\frac{\left( w^{\prime }\right)
^{2}}{w}\left( r\right) dr+\frac{4\pi }{t}.  \notag
\end{eqnarray}%
From the elementary inequality

\begin{eqnarray}
0 &\leq &w\left( \frac{w^{\prime }}{w}-\frac{2}{r}\right) ^{2}  \label{w} \\
&=&\frac{\left( w^{\prime }\right) ^{2}}{w}-\frac{4}{r}w^{\prime }+\frac{4}{%
r^{2}}w,  \notag
\end{eqnarray}%
one sees that

\begin{equation*}
-\frac{3}{4}\int_{t}^{\infty }r^{-2}\frac{\left( w^{\prime }\right) ^{2}}{w}%
\left( r\right) dr\leq -3\int_{t}^{\infty }r^{-3}w^{\prime }\left( r\right)
dr+3\int_{t}^{\infty }r^{-4}w\left( r\right) dr.
\end{equation*}%
Furthermore, integrating by parts implies that

\begin{eqnarray*}
-3\int_{t}^{\infty }r^{-3}w^{\prime }\left( r\right) dr &=&-3r^{-3}w\left(
r\right) |_{t}^{\infty }-9\int_{t}^{\infty }r^{-4}w\left( r\right) dr \\
&=&3t^{-3}w\left( t\right) -9\int_{t}^{\infty }r^{-4}w\left( r\right) dr.
\end{eqnarray*}%
In conclusion,

\begin{equation}
-\frac{3}{4}\int_{t}^{\infty }r^{-2}\frac{\left( w^{\prime }\right) ^{2}}{w}%
\left( r\right) dr\leq 3t^{-3}w\left( t\right) -6\int_{t}^{\infty
}r^{-4}w\left( r\right) dr.  \label{m7.2}
\end{equation}%
Plugging (\ref{m7.2}) into (\ref{m7.1}) we get

\begin{equation}
-\int_{L\left( t,\infty \right) }G^{-2}\Delta \left\vert \nabla G\right\vert
\leq 3t^{-3}w\left( t\right) -6\int_{t}^{\infty }r^{-4}w\left( r\right) dr+%
\frac{4\pi }{t}.  \label{m7.3}
\end{equation}

Hence, by (\ref{m7.3}) and (\ref{m3}) we obtain

\begin{equation*}
t^{-2}\frac{dw}{dt}\left( t\right) \leq t^{-3}w\left( t\right) +\frac{4\pi }{%
t},
\end{equation*}%
or equivalently,

\begin{equation}
\frac{d}{dt}\left( \frac{1}{t}w\left( t\right) -4\pi t\right) \leq 0.
\label{m8}
\end{equation}%
Finally, if

\begin{equation*}
\frac{d}{dt}\left( \frac{1}{t}w\left( t\right) -4\pi t\right) =0\text{ for }
t=T,
\end{equation*}%
then all the inequalities in Lemma \ref{Kato}, (\ref{w}) and (\ref{m6'})
become equality. In particular, the Hessian of $G$ on $L\left( T,\infty
\right) $ must be of the form

\begin{eqnarray*}
G_{11} &=&32\pi ^{2}G^{3} \\
G_{22} &=&-16\pi ^{2}G^{3} \\
G_{33} &=&-16\pi ^{2}G^{3} \\
G_{ij} &=&0\;\;\text{ otherwise}
\end{eqnarray*}%
and $\left\vert \nabla G\right\vert =4\pi G^{2}.$ Now consider the function $%
f(x)=\frac{1}{4\pi\,G(x)}.$ Then $\left\vert \nabla f\right\vert=1$ and the
Hessian of the function $\frac{1}{2}\,f^2$ is the identity matrix. This
immediately implies that $L\left( T,\infty \right)$ is isometric to the ball 
$B_0(\frac{1}{4\pi\,T})$ in $\mathbb{R}^3.$
\end{proof}

We conclude this section with the following corollary.

\begin{corollary}
\label{D1} Let $\left( M,g\right) $ be a complete noncompact
three-dimensional manifold with nonnegative scalar curvature and
asymptotically nonnegative Ricci curvature, that is, 
\begin{equation}
\liminf_{x\rightarrow \infty }\mathrm{Ric}\left( x\right) \geq 0.  \label{x2}
\end{equation}%
Assume that $M$ has one end and its first Betti number $b_{1}\left( M\right)
=0.$ If $M$ is nonparabolic and the minimal Green's function $G\left(
x\right) =G\left( p,x\right) $ satisfies $\lim_{x\rightarrow \infty }G(x)=0,$
then

\begin{equation*}
\int_{l\left( t\right) }\left\vert \nabla G\right\vert ^{2}\leq 4\pi t^{2}
\end{equation*}%
and 
\begin{equation*}
\mathrm{Area}\left( l\left( t\right) \right) \geq \frac{1}{4\pi t^{2}}
\end{equation*}%
for all $t>0.$ Moreover, if equality holds for some $T>0,$ then the super
level set $\left\{ G>T\right\} $ is isometric to a ball in $\mathbb{R}^3.$
\end{corollary}

\begin{proof}
Let us note that by Theorem \ref{Positive} we have

\begin{equation}
\frac{1}{t}w\left( t\right) -4\pi t\leq \frac{1}{\delta }w\left( \delta
\right) -4\pi \delta  \label{x1}
\end{equation}%
for all $0<\delta <t.$ Now the gradient estimate in \cite{CY} together with
the assumption (\ref{x2}) implies that for any $\varepsilon >0$ there exists
sufficiently small $\delta >0$ such that

\begin{equation}
\left\vert \nabla \ln G\right\vert \leq \varepsilon \text{ \ on }L\left(
0,\delta \right).  \label{x3}
\end{equation}%
Therefore,

\begin{equation*}
\frac{1}{\delta }w\left( \delta \right) =\frac{1}{\delta }\int_{l\left(
\delta \right) }\left\vert \nabla G\right\vert ^{2}\leq \varepsilon,
\end{equation*}%
where we have used the fact that 
\begin{equation}
\int_{l\left( \delta\right) }\left\vert \nabla G\right\vert =1.  \label{a0}
\end{equation}%
This shows that the right hand side of (\ref{x1}) goes to $0$ as $\delta\to
0.$ Hence, $w(t)\leq 4\pi t^2,$ or

\begin{equation*}
\int_{l\left( t\right) }\left\vert \nabla G\right\vert ^{2}\leq 4\pi t^{2}
\end{equation*}%
for all $t>0.$

We now derive a sharp area estimate for the level sets of the Green's
function. Indeed,

\begin{equation*}
1=\int_{l\left( t\right) }\left\vert \nabla G\right\vert \leq \left(
\int_{l\left( t\right) }\left\vert \nabla G\right\vert ^{2}\right) ^{\frac{1%
}{2}}\left( \mathrm{Area}\left( l\left( t\right) \right) \right) ^{\frac{1}{2%
}}.
\end{equation*}%
It follows that

\begin{equation*}
\mathrm{Area}\left( l\left( t\right) \right) \geq \frac{1}{4\pi t^{2}}
\end{equation*}
for all $t>0.$

Obviously, the equality case follows from that of Theorem \ref{Positive}.
\end{proof}

We note that under the hypothesis of Corollary \ref{D1} if

\begin{equation}
\limsup_{t\rightarrow \infty} \left( \frac{1}{t}\int_{l\left( t\right)
}\left\vert \nabla G\right\vert ^{2}-4\pi t\right) \geq 0,  \label{x4}
\end{equation}%
then $\left( M,g\right) $ is isometric to $\mathbb{R}^{3}.$ Indeed, by (\ref%
{x1}) and (\ref{x4}) we have

\begin{equation*}
0\leq \frac{1}{\delta }\int_{l\left( \delta \right) }\left\vert \nabla
G\right\vert ^{2}-4\pi \delta
\end{equation*}%
for any $\delta >0.$ Hence, equality must hold in Corollary \ref{D1} and $%
\left( M,g\right) $ is isometric to the Euclidean space $\mathbb{R}^3.$

\section{Negative scalar lower bound \label{sect4}}

We now turn to the proof of Theorem \ref{T2}. We start by establishing some
lemmas under the assumption that $\left( M,g\right)$ admits a positive
Green's function $G$ satisfying

\begin{eqnarray}
\lim_{x\rightarrow \infty }G\left( x\right) &=&0,  \label{g1} \\
\#\mathrm{Conn}\left( l\left( t\right) \right) &\leq &A  \label{g2}
\end{eqnarray}%
for all $t\leq t_0,$ where $t_0$ and $A>0$ are fixed constants, and $\#%
\mathrm{Conn}\left(l(t) \right) $ denotes the number of connected components
of the level set $l(t)$ of $G.$

\begin{lemma}
\label{T} Let $\left( M,g\right) $ be a three-dimensional complete
noncompact Riemannian manifold satisfying (\ref{g1}) and (\ref{g2}). Assume
that the Ricci curvature of $M$ is bounded from below and the scalar
curvature $S$ is bounded by 
\begin{equation*}
S\geq -6K\text{ \ on }M.
\end{equation*}%
Then for any $\varepsilon >0,$

\begin{equation*}
\int_{L\left( \varepsilon, t_{0}\right) }\left\vert \nabla \left\vert \nabla
G\right\vert \right\vert ^{2}\left\vert \nabla G\right\vert ^{-1}\leq
4K\int_{L\left( \frac{1}{2}\varepsilon ,2t_{0}\right) }\left\vert \nabla
G\right\vert +\frac{32\pi }{3}t_{0}A+C,
\end{equation*}%
where $C$ is a constant depending only on $t_{0}$ and the Ricci curvature
lower bound of $M,$ but not $\varepsilon.$
\end{lemma}

\begin{proof}
According to the Bochner formula,

\begin{eqnarray*}
\frac{1}{2}\Delta \left\vert \nabla G\right\vert ^{2} &=&\left\vert
G_{ij}\right\vert ^{2}+\left\langle \nabla \Delta G,\nabla G\right\rangle +%
\mathrm{Ric}\left( \nabla G,\nabla G\right) \\
&=&\left\vert G_{ij}\right\vert ^{2}+\mathrm{Ric}\left( \nabla G,\nabla
G\right)
\end{eqnarray*}%
on $M\backslash \left\{ p\right\}.$ Therefore,

\begin{equation}
\Delta \left\vert \nabla G\right\vert =\left( \left\vert G_{ij}\right\vert
^{2}-\left\vert \nabla \left\vert \nabla G\right\vert \right\vert
^{2}\right) \left\vert \nabla G\right\vert ^{-1}+\mathrm{Ric}\left( \nabla
G,\nabla G\right) \left\vert \nabla G\right\vert ^{-1}  \label{b1}
\end{equation}%
holds on $M\backslash \left\{ p\right\} $ whenever $\left\vert \nabla
G\right\vert \neq 0.$

Fix $0<\varepsilon <t_{0}<\infty $ and let $\phi $ be the Lipschitz function
with support in $L\left( \frac{1}{2}\varepsilon ,2t_{0}\right) $ defined by

\begin{equation}
\phi =\left\{ 
\begin{array}{c}
1 \\ 
\frac{\ln G-\ln \left( \frac{1}{2}\varepsilon \right) }{\ln 2} \\ 
\frac{\ln \left( 2t_{0}\right) -\ln G}{\ln 2} \\ 
0%
\end{array}%
\right. 
\begin{array}{c}
\text{on }L\left( \varepsilon ,1\right) \\ 
\text{on }L\left( \frac{1}{2}\varepsilon ,\varepsilon \right) \\ 
\text{on }L\left( t_{0},2t_{0}\right) \\ 
\text{otherwise}%
\end{array}
\label{b2}
\end{equation}%
By the co-area formula, we have

\begin{eqnarray*}
&&\int_{M}\left( \left\vert G_{ij}\right\vert ^{2}-\left\vert \nabla
\left\vert \nabla G\right\vert \right\vert ^{2}+\mathrm{Ric}\left( \nabla
G,\nabla G\right) \right) \left\vert \nabla G\right\vert ^{-1}\phi ^{2} \\
&=&\int_{\frac{1}{2}\varepsilon }^{2t_{0}}\phi ^{2}\left( r\right)
\int_{l\left( r\right) }\left( \left\vert G_{ij}\right\vert ^{2}-\left\vert
\nabla \left\vert \nabla G\right\vert \right\vert ^{2}+\mathrm{Ric}\left(
\nabla G,\nabla G\right) \right) \left\vert \nabla G\right\vert ^{-2}dr.
\end{eqnarray*}

However, Lemma \ref{RicS} says that

\begin{eqnarray*}
&&\left( \left\vert G_{ij}\right\vert ^{2}-\left\vert \nabla \left\vert
\nabla G\right\vert \right\vert ^{2}+\mathrm{Ric}\left( \nabla G,\nabla
G\right) \right) \left\vert \nabla G\right\vert ^{-2} \\
&=&\frac{1}{2}S-\frac{1}{2}S_{l\left( r\right) }+\frac{1}{2}\left\vert
G_{ij}\right\vert ^{2}\left\vert \nabla G\right\vert ^{-2}.
\end{eqnarray*}%
Applying Lemma \ref{Kato} for the last term, we conclude that

\begin{eqnarray}
&&\left( \left\vert G_{ij}\right\vert ^{2}-\left\vert \nabla \left\vert
\nabla G\right\vert \right\vert ^{2}+\mathrm{Ric}\left( \nabla G,\nabla
G\right) \right) \left\vert \nabla G\right\vert ^{-2}  \label{b4} \\
&\geq &\frac{1}{2}S-\frac{1}{2}S_{l\left( r\right) }+\frac{3}{4}\left\vert
\nabla \left\vert \nabla G\right\vert \right\vert ^{2}\left\vert \nabla
G\right\vert ^{-2}.  \notag
\end{eqnarray}%
According to the Gauss-Bonnet theorem, on each regular connected component $%
l_{k}\left( r\right) $ of $l\left( r\right),$

\begin{equation*}
\int_{l_{k}\left( r\right) }S_{l\left( r\right) }=4\pi \chi \left(
l_{k}\left( r\right) \right) \leq 8\pi
\end{equation*}%
as $l_{k}\left( r\right) $ is compact by (\ref{g1}). Since by hypothesis (%
\ref{g2}) there are at most $A$ connected components of $l\left(r\right),$
it follows that

\begin{equation*}
\int_{l\left( r\right) }S_{l\left( r\right) }\leq 8\pi A
\end{equation*}%
for all regular value $r$ with $r\leq t_{0}.$ Therefore, using that $S\geq
-6K,$ we conclude from (\ref{b4}) that

\begin{eqnarray*}
&&\int_{l\left( r\right) }\left( \left\vert G_{ij}\right\vert
^{2}-\left\vert \nabla \left\vert \nabla G\right\vert \right\vert ^{2}+%
\mathrm{Ric}\left( \nabla G,\nabla G\right) \right) \left\vert \nabla
G\right\vert ^{-2} \\
&\geq &\int_{l\left( r\right) }\left( \frac{1}{2}S-\frac{1}{2}S_{l\left(
r\right) }+\frac{3}{4}\left\vert \nabla \left\vert \nabla G\right\vert
\right\vert ^{2}\left\vert \nabla G\right\vert ^{-2}\right) \\
&\geq &-3K\mathrm{Area}\left( l\left( r\right) \right) +\frac{3}{4}%
\int_{l\left( r\right) }\left\vert \nabla \left\vert \nabla G\right\vert
\right\vert ^{2}\left\vert \nabla G\right\vert ^{-2}-4\pi A.
\end{eqnarray*}%
Consequently, this implies that

\begin{eqnarray*}
&&\int_{M}\left( \left\vert G_{ij}\right\vert ^{2}-\left\vert \nabla
\left\vert \nabla G\right\vert \right\vert ^{2}+\mathrm{Ric}\left( \nabla
G,\nabla G\right) \right) \left\vert \nabla G\right\vert ^{-1}\phi ^{2} \\
&\geq &\int_{\frac{1}{2}\varepsilon }^{2t_{0}}\phi ^{2}\left( r\right)
\left( -3K\mathrm{Area}\left( l\left( r\right) \right) +\frac{3}{4}%
\int_{l\left( r\right) }\left\vert \nabla \left\vert \nabla G\right\vert
\right\vert ^{2}\left\vert \nabla G\right\vert ^{-2}\right) dr-8\pi t_{0}A \\
&=&-3K\int_{M}\left\vert \nabla G\right\vert \phi ^{2}+\frac{3}{4}%
\int_{M}\left\vert \nabla \left\vert \nabla G\right\vert \right\vert
^{2}\left\vert \nabla G\right\vert ^{-1}\phi ^{2}-8\pi t_{0}A.
\end{eqnarray*}%
Together with the Bochner formula (\ref{b1}), we arrive at

\begin{equation}
-3K\int_{M}\left\vert \nabla G\right\vert \phi ^{2}+\frac{3}{4}%
\int_{M}\left\vert \nabla \left\vert \nabla G\right\vert \right\vert
^{2}\left\vert \nabla G\right\vert ^{-1}\phi ^{2}\leq \int_{M}\phi
^{2}\Delta \left\vert \nabla G\right\vert +8\pi t_{0}A.  \label{b5}
\end{equation}%
To estimate the right hand side, we make use of the gradient estimate for
positive harmonic functions. Choose $R_{0}>0$ so that $L\left(0,2t_{0}%
\right) \subset M\backslash B\left( p,R_{0}\right),$ where $p$ is the pole
of $G.$ Since the Ricci curvature is bounded from below on $M,$ applying
Cheng-Yau's gradient estimate \cite{CY} to the positive harmonic function $G$
on $M\backslash B\left( p,\frac{R_{0}}{2}\right),$ one concludes that

\begin{equation}
\left\vert \nabla \ln G\right\vert \leq \Lambda \text{ \ on }M\backslash
B\left( p,R_{0}\right),  \label{CY}
\end{equation}%
where the constant $\Lambda $ depends only on $R_{0}$ and Ricci curvature
lower bound of $M.$

Integration by parts implies that

\begin{eqnarray*}
\int_{M}\phi ^{2}\Delta \left\vert \nabla G\right\vert
&=&-\int_{M}\left\langle \nabla \phi ^{2},\nabla \left\vert \nabla
G\right\vert \right\rangle \\
&=&-\int_{L\left( \frac{1}{2}\varepsilon ,\varepsilon \right) }\left\langle
\nabla \phi ^{2},\nabla \left\vert \nabla G\right\vert \right\rangle \\
&&-\int_{L\left( t_{0},2t_{0}\right) }\left\langle \nabla \phi ^{2},\nabla
\left\vert \nabla G\right\vert \right\rangle.
\end{eqnarray*}%
Further integration by parts on each term leads to

\begin{eqnarray*}
-\int_{L\left( \frac{1}{2}\varepsilon ,\varepsilon \right) }\left\langle
\nabla \phi ^{2},\nabla \left\vert \nabla G\right\vert \right\rangle 
&=&\int_{L\left( \frac{1}{2}\varepsilon ,\varepsilon \right) }\left\vert
\nabla G\right\vert \Delta \phi ^{2}-\int_{l\left( \varepsilon \right)
}\left\langle \nabla \phi ^{2},\nabla G\right\rangle  \\
-\int_{L\left( t_{0},2t_{0}\right) }\left\langle \nabla \phi ^{2},\nabla
\left\vert \nabla G\right\vert \right\rangle  &=&\int_{L\left(
t_{0},2t_{0}\right) }\left\vert \nabla G\right\vert \Delta \phi
^{2}+\int_{l\left( t_{0}\right) }\left\langle \nabla \phi ^{2},\nabla
G\right\rangle .
\end{eqnarray*}%
Noting that $G$ is harmonic, we get that on $L\left( \frac{1}{2}\varepsilon
,\varepsilon \right) $

\begin{eqnarray*}
\Delta \phi ^{2} &=&2\phi \Delta \phi +2\left\vert \nabla \phi \right\vert
^{2} \\
&=&2\left( -\frac{1}{\ln 2}\phi +\frac{1}{\left( \ln 2\right) ^{2}}\right)
\left\vert \nabla \ln G\right\vert ^{2}.
\end{eqnarray*}%
Similarly, on $L\left( t_{0},2t_{0}\right) ,$

\begin{equation*}
\Delta \phi ^{2}=2\left( \frac{1}{\ln 2}\phi +\frac{1}{\left( \ln 2\right)
^{2}}\right) \left\vert \nabla \ln G\right\vert ^{2}.
\end{equation*}%
In both cases, in view of (\ref{CY}), we have

\begin{eqnarray*}
\left\vert \Delta \phi ^{2}\right\vert &\leq &c\left\vert \nabla \ln
G\right\vert ^{2} \\
&\leq &c\Lambda \left\vert \nabla \ln G\right\vert,
\end{eqnarray*}%
where $c$ is a universal constant. Hence, by the co-area formula and the
fact that 
\begin{equation}
\int_{l\left( r\right) }\left\vert \nabla G\right\vert =1,  \label{lev}
\end{equation}
we get

\begin{eqnarray*}
\left\vert \int_{L\left( \frac{1}{2}\varepsilon ,\varepsilon \right)
}\left\vert \nabla G\right\vert \Delta \phi ^{2}\right\vert  &\leq &c\Lambda
\int_{L\left( \frac{1}{2}\varepsilon ,\varepsilon \right) }\left\vert \nabla
G\right\vert ^{2}G^{-1} \\
&=&c\Lambda \int_{\frac{1}{2}\varepsilon }^{\varepsilon }\frac{1}{r}dr \\
&=&c\Lambda \ln 2.
\end{eqnarray*}%
The other term is estimated as follows.

\begin{eqnarray*}
\left\vert \int_{l\left( \varepsilon \right) }\left\langle \nabla \phi
^{2},\nabla G\right\rangle \right\vert  &\leq &\frac{2}{\ln 2}\int_{l\left(
\varepsilon \right) }\left\vert \nabla G\right\vert ^{2}G^{-1} \\
&\leq &c\Lambda .
\end{eqnarray*}%
Similarly, 
\begin{equation*}
\left\vert \int_{L\left( t_{0},2t_{0}\right) }\left\vert \nabla G\right\vert
\Delta \phi ^{2}\right\vert +\left\vert \int_{l\left( t_{0}\right)
}\left\langle \nabla \phi ^{2},\nabla G\right\rangle \right\vert \leq
c\Lambda .
\end{equation*}%
In conclusion, we have shown that

\begin{equation}
\left\vert \int_{M}\phi ^{2}\Delta \left\vert \nabla G\right\vert
\right\vert \leq c\Lambda \text{\ }\ \text{and }\int_{M}\left\vert \nabla
\phi \right\vert ^{2}\left\vert \nabla G\right\vert \leq c\Lambda.
\label{bdry}
\end{equation}

Plugging into (\ref{b5}) implies that

\begin{equation*}
\frac{3}{4}\int_{M}\left\vert \nabla \left\vert \nabla G\right\vert
\right\vert ^{2}\left\vert \nabla G\right\vert ^{-1}\phi ^{2}\leq
3K\int_{M}\left\vert \nabla G\right\vert \phi ^{2}+c\Lambda +8\pi t_{0}A,
\end{equation*}%
which is what to be proved.
\end{proof}

\begin{lemma}
\label{Main}Let $\left( M,g\right) $ be a three-dimensional complete
noncompact Riemannian manifold satisfying (\ref{g1}) and (\ref{g2}). Assume
that the Ricci curvature of $M$ is bounded from below and the scalar
curvature $S$ is bounded by 
\begin{equation*}
S\geq -6K\text{ \ on }M.
\end{equation*}%
Then 
\begin{equation*}
\lambda _{1}\left( M\right) \leq K.
\end{equation*}
\end{lemma}

\begin{proof}
According to Lemma \ref{T},

\begin{equation}
\frac{3}{4}\int_{M}\left\vert \nabla \left\vert \nabla G\right\vert
\right\vert ^{2}\left\vert \nabla G\right\vert ^{-1}\phi ^{2}\leq
3K\int_{M}\left\vert \nabla G\right\vert \phi ^{2}+c\Lambda +8\pi t_{0}A,
\label{b6}
\end{equation}%
where $\phi $ is given in (\ref{b2}).

We now bound the left hand side from below using the bottom spectrum $%
\lambda_1(M).$ Note that 
\begin{eqnarray*}
\lambda _{1}\left( M\right) \int_{M}\left\vert \nabla G\right\vert \phi ^{2}
&=&\lambda _{1}\left( M\right) \int_{M}\left( \left\vert \nabla G\right\vert
^{\frac{1}{2}}\phi \right) ^{2} \\
&\leq &\int_{M}\left\vert \nabla \left( \left\vert \nabla G\right\vert ^{%
\frac{1}{2}}\phi \right) \right\vert ^{2}.
\end{eqnarray*}%
Expanding the right side, we get

\begin{eqnarray*}
\int_{M}\left\vert \nabla \left( \left\vert \nabla G\right\vert ^{\frac{1}{2}%
}\phi \right) \right\vert ^{2} &=&\frac{1}{4}\int_{M}\left\vert \nabla
\left\vert \nabla G\right\vert \right\vert ^{2}\left\vert \nabla
G\right\vert ^{-1}\phi ^{2}+\int_{M}\left\vert \nabla G\right\vert
\left\vert \nabla \phi \right\vert ^{2} \\
&&+\frac{1}{2}\int_{M}\left\langle \nabla \left\vert \nabla G\right\vert
,\nabla \phi ^{2}\right\rangle.
\end{eqnarray*}%
By (\ref{bdry}) we have

\begin{equation*}
\int_{M}\left\vert \nabla G\right\vert \left\vert \nabla \phi \right\vert
^{2}+\frac{1}{2}\int_{M}\left\langle \nabla \left\vert \nabla G\right\vert
,\nabla \phi ^{2}\right\rangle \leq c\Lambda.
\end{equation*}%
Therefore,

\begin{equation*}
\int_{M}\left\vert \nabla \left( \left\vert \nabla G\right\vert ^{\frac{1}{2}%
}\phi \right) \right\vert ^{2}\leq \frac{1}{4}\int_{M}\left\vert \nabla
\left\vert \nabla G\right\vert \right\vert ^{2}\left\vert \nabla
G\right\vert ^{-1}\phi ^{2}+c\Lambda.
\end{equation*}%
In conclusion,

\begin{eqnarray}
\lambda _{1}\left( M\right) \int_{M}\left\vert \nabla G\right\vert \phi ^{2}
&\leq &\int_{M}\left\vert \nabla \left( \left\vert \nabla G\right\vert ^{%
\frac{1}{2}}\phi \right) \right\vert ^{2}  \label{b7} \\
&\leq &\frac{1}{4}\int_{M}\left\vert \nabla \left\vert \nabla G\right\vert
\right\vert ^{2}\left\vert \nabla G\right\vert ^{-1}\phi ^{2}+c\Lambda. 
\notag
\end{eqnarray}%
Combining (\ref{b6}) and (\ref{b7}) we arrive at the following inequality.

\begin{equation}
\lambda _{1}\left( M\right) \int_{M}\left\vert \nabla G\right\vert \phi
^{2}\leq K\int_{M}\left\vert \nabla G\right\vert \phi ^{2}+c\Lambda +\frac{%
8\pi }{3}t_{0}A,  \label{b8}
\end{equation}%
where $c$ is a universal constant. Assume by contradiction that

\begin{equation}
\lambda _{1}\left( M\right) >K.  \label{b11}
\end{equation}%
Then in view of the definition of $\phi $ in (\ref{b2}) and (\ref{b8}) we
conclude that

\begin{equation}
\left( \lambda _{1}\left( M\right) -K\right) \int_{L\left( \varepsilon
,t_{0}\right) }\left\vert \nabla G\right\vert \leq c\left( \Lambda +A\right).
\label{b12}
\end{equation}%
However, applying (\ref{lev}) and the co-area formula, we have

\begin{eqnarray*}
\int_{L\left( \varepsilon ,t_{0}\right) }\left\vert \nabla G\right\vert
^{2}G^{-1} &=&\int_{\varepsilon }^{t_{0}}\frac{1}{r}\left( \int_{l\left(
r\right) }\left\vert \nabla G\right\vert \right) dr \\
&=&\ln \left( t_{0}\varepsilon ^{-1}\right).
\end{eqnarray*}%
On the other hand, Cheng-Yau's gradient estimate shows that

\begin{equation*}
\int_{L\left( \varepsilon ,t_{0}\right) }\left\vert \nabla G\right\vert
^{2}G^{-1}\leq \Lambda \int_{L\left( \varepsilon ,t_{0}\right) }\left\vert
\nabla G\right\vert.
\end{equation*}%
We thus conclude that

\begin{equation}
\int_{L\left( \varepsilon ,t_{0}\right) }\left\vert \nabla G\right\vert \geq
\Lambda ^{-1}\ln \left( t_{0}\varepsilon ^{-1}\right).  \label{b13}
\end{equation}%
Finally, we infer from (\ref{b13}) and (\ref{b12}) that

\begin{equation}
\lambda _{1}\left( M\right) -K\leq \frac{c\Lambda \left( \Lambda +A\right) }{%
\ln \left( t_{0}\varepsilon ^{-1}\right) }.  \label{b14}
\end{equation}%
Note that $c$ is a universal constant and that both $A$ and $\Lambda $ are
independent of $\varepsilon.$ Taking $\varepsilon \rightarrow 0$ leads to a
contradiction. Therefore, we have

\begin{equation*}
\lambda _{1}\left( M\right) \leq K.
\end{equation*}
\end{proof}

We are now ready to prove Theorem \ref{T2}. The only thing left to do is to
verify that both (\ref{g1}) and (\ref{g2}) hold.

\begin{theorem}
\label{Main2} Let $\left( M,g\right) $ be a three-dimensional complete
noncompact Riemannian manifold with scalar curvature $S\geq -6K$ on $M$ for
some nonnegative constant $K.$ Suppose that $M$ has finitely many ends and
its first Betti number $b_{1}(M)<\infty .$ Moreover, the Ricci curvature of $%
M$ is bounded below and the volume $V_{x}(1)$ of unit ball $B_{x}(1)$
satisfies

\begin{equation*}
V_x(1)\geq C(\epsilon)\,\exp\left(-2\sqrt{K+\epsilon}\,r(x)\right)
\end{equation*}%
for every $\epsilon>0$ and all $x\in M,$ where $r(x)$ is the geodesic
distance from $x$ to a fixed point $p.$ Then the bottom spectrum of the
Laplacian satisfies 
\begin{equation*}
\lambda _{1}\left( M\right) \leq K.
\end{equation*}
\end{theorem}

\begin{proof}
We prove the result by contradiction. Suppose that $\lambda _{1}\left(
M\right) >K.$ Then $M$ is necessarily nonparabolic as $\lambda _{1}(M)>0$,
see \cite[Chapter 22]{L}. Let $G$ be the minimal positive Green's function.
We first show that $\lim_{x\rightarrow \infty }G(x)=0$ or (\ref{g1}) holds.

According to a result of Li and the second author \cite{LW},

\begin{equation}
\int_{B_p\left(R+1\right) \backslash B_p\left(R-1\right) }G^{2}\leq C\,e^{-2%
\sqrt{\lambda _{1}\, \left( M\right) }R}\int_{B_p\left(2\right) \backslash
B_p\left(1\right) }G^{2}  \label{e1}
\end{equation}%
for all $R>4,$ where $C>0$ is a constant depending only on $%
\lambda_{1}\left( M\right).$ Since the Ricci curvature is bounded from
below, by (\ref{CY}) the mean value inequality holds for function $G,$ that
is, for any $x\in M\setminus B_p(2),$

\begin{equation}
G^{2}\left( x\right) \leq \frac{C}{V_x\left(1\right)}\,\int_{B\left(
x,1\right) }G^{2}  \label{e2}
\end{equation}
for some constant $C$ only depending on the Ricci curvature lower bound. Now
fix $\epsilon>0$ such that $\lambda_1(M)\geq K+2\epsilon.$ Combining (\ref%
{e1}) and (\ref{e2}), we conclude that

\begin{equation}
G^{2}\left( x\right) \leq \frac{C}{V_x\left(1\right)}\, e^{-2\sqrt{%
K+2\epsilon}\,r(x)}\,\int_{B_p\left(2\right) \backslash B_p\left(1\right)
}G^{2} ,  \label{e3}
\end{equation}%
where $r(x)$ is the geodesic distance from $x$ to point $p.$ Using the
assumption that

\begin{equation*}
V_x(1)\geq C(\epsilon)\,\exp\left(-2\sqrt{K+\epsilon}\,r(x)\right),
\end{equation*}%
one sees immediately from (\ref{e3}) that $\lim_{x\to \infty}G(x)=0.$

The fact that $G$ satisfies (\ref{g2}) follows directly from Lemma \ref{LT}.
Therefore, Lemma \ref{Main} is applicable and $\lambda_1(M)\leq K.$
\end{proof}

We now draw a corollary.

\begin{corollary}
Let $\left( M,g\right) $ be a three-dimensional complete noncompact
Riemannian manifold with nonpositive sectional curvature. Assume that the
scalar curvature $S$ is bounded by

\begin{equation*}
S\geq -6K\text{ \ on }M
\end{equation*}
for some nonnegative constant $K.$ Then 
\begin{equation*}
\lambda _{1}\left( M\right) \leq K.
\end{equation*}
\end{corollary}

\begin{proof}
This is because Theorem \ref{Main2} is applicable to the universal cover $%
\tilde{M}$ of $M$ as $\tilde{M}$ is a Cartan-Hadamard manifold with bounded
curvature. Indeed, being diffeomorphic to $\mathbb{R}^3,$ $\tilde{M}$ has
first Betti number $0$ and exactly one end. Also, the volume $V_x(1)$ of
unit balls $B_x(1)$ is at least that of the unit ball in $\mathbb{R}^3$ by
the volume comparison theorem. Therefore, $\lambda_1(\tilde{M})\leq K.$
However, the bottom spectrum satisfies $\lambda_1(M)\leq \lambda_1(\tilde{M}%
),$ the corollary follows.
\end{proof}

\textbf{Acknowledgment}. The first author was partially supported by NSF
grant DMS-1811845.

\end{document}